\documentclass[12pt]{amsart}
\usepackage{amsmath,amssymb,fullpage}

\usepackage{geometry}
\usepackage{graphicx}
\usepackage{amssymb}
\usepackage{mathrsfs}
\usepackage{enumerate}
\usepackage{amsfonts}
\usepackage{enumitem}
\usepackage{amsmath}
\usepackage{cases}
\usepackage{amsfonts}
\usepackage{amsmath}
\usepackage{fancyhdr}
\usepackage[all]{xy}
\title[On the Maximum of the Permanent of $I-A$]{On the Maximum of the Permanent of $(I-A)$}

\author{Zhi Chen}
\address{Department of Mathematics, Nanjing Agricultural University, Jiangsu, 210095, China}
\email {chenzhi@njau.edu.cn }

\author{Lei Cao}
\address{Department of Mathematics, Georgian Court University, Lakewood, NJ 08701, USA}
\email {lcao@georgian.edu}

\thanks{Z. Chen is supported by the Natural Science Foundation of Jiangsu Province (BK20160708); the Fundamental Research Funds for the Central Universities (No.KJQN201718); the National Natural Science Foundation of China (No.11601233).}


\theoremstyle{plain}

\newtheorem{thm}{Theorem}[section]
\newtheorem{cor}[thm]{Corollary}
\newtheorem{lem}[thm]{Lemma}

\newtheorem{prop}[thm]{Proposition}
\newtheorem{conj}[thm]{Conjecture}
\theoremstyle{remark}
\newtheorem{remark}{Remark}
\theoremstyle{question}

\numberwithin{equation}{section}

\newcommand{\beq}{\begin{equation}}
\newcommand{\eeq}{\end{equation}}

\theoremstyle{remark}

\numberwithin{equation}{section}

\newcommand{\bbm}{\begin{bmatrix}}
\newcommand{\ebm}{\end{bmatrix}}

\begin{document}

\maketitle

\begin{abstract}
Let $A$ be an $n\times n$ doubly substochastic matrix and denote by $\sigma(A)$ the sum of all elements of $A.$  In this paper we give the upper bound of the permanent of $(I-A)$ with respect to $n$ and $\sigma(A).$
\end{abstract}

%
%

\section{Introduction}
Let $A = [a_{ij}]$ be an $n \times n$ matrix and $S_n$ be the symmetric group of order $n.$
The {\it permanent} of $A$ is the scalar-valued function of A defined by
$$ {\rm per(A)}=\sum_{\pi\in S_n}a_{1\pi(1)}a_{2\pi(2)}\ldots a_{n\pi(n)},$$
where the summation extends over all $n!$ permutations in $S_n.$ This function has been studied intensively (see \cite{per}, \cite{theory1} and \cite{theory2}) and it appears naturally in many combinatorial settings where a count of the number of systems of distinct representatives of some configuration is required \cite{ryser_1963}.

An $n\times n$ nonnegative matrix is said to be a {\it doubly stochastic matrix} if the sum of each row and column is one. If we allow the sum of each row and column to be less than or equal to one, then we get a {\it doubly substochastic matrix}. Denote by $\Omega_n$ the set of all $n\times n$ doubly stochastic matrices and $\omega_n$ the set of all $n\times n$ doubly substochastic matrices. Let $I$ be the $n\times n $ identity matrix. In this paper, we focus on per$(I-A)$ for $A\in \omega_n$, which was brought to people's attention by Marcus and Minc in 1965 who conjectured the following lower bound of per$(I-A)$:
\begin{equation} \label{ineq1}
 {\rm per}(I-A)\geq 0
\end{equation}
for all $A\in \Omega_n$ (See Conjecture 7 in \cite{per}).
It was firstly solved by Brualdi and Newman~\cite{BN1966} who showed that \eqref{ineq1} is true in a more general case when $A$ is a row substochastic matrix. A row substochastic matrix, or sometimes a substochastic matrix, is a nonnegative matrix with all row sums less than or equal to one. We denote the set of all $n\times n$ row substochastic matrices by $\tilde{\omega}_n$. Gibson then gave another short proof in~\cite{GB2} and later improved~\eqref{ineq1} in~\cite{GB1} to the following inequality:
\begin{equation} \label{ineq2}
{\rm per}(I-A)\geq\det(I-A) \geq 0
\end{equation}
for $A\in\tilde{\omega}_n$. The upper bound was given by Malek~\cite{MM} who showed that
\begin{equation} \label{ineq3}
{\rm per}(I-A)\leq 2^{\lfloor\frac{n}{2}\rfloor}
\end{equation}
for $A\in\tilde{\omega}_n$, where $\lfloor x\rfloor$ takes the greatest integer less than or equal to $x$.
Since a doubly substochastic matrix is surely to be row substochastic,  \eqref{ineq3} also provides an upper bound for all matrices in $\omega_n$.
However, \eqref{ineq3} is not that accurate regarding to the row substochastic matrices with summation of all elements far less than its size. 
Therefore in this paper we explore a finer upper bound of ${\rm per}(I-A)$  for a doubly substochastic matrix $A$ with fixed summation of all elements. To do this,  we partition $\omega_n$ with respect to the sum of all elements in the matrices.
Let $A=[a_{ij}]$ be an $n$-square matrix and denote by $\sigma(A)$ the sum of all elements $\sum_{i,j=1}^{n}a_{i j}$. It is easy to see that for any $0\leq s\leq n,$ the sets
$$\omega_n^s:=\{A\in \omega_n \ |\ \sigma(A)=s \} \ {\rm and} \  \tilde{\omega}_n^s:=\{A\in \tilde{\omega}_n \ |\ \sigma(A)=s \}$$
are convex. Moreover in this paper we shall give an upper bound for both of the sets
\begin{equation*}
\{{\rm per} (I-A)|A\in \omega_n^s\}\ {\rm and} \  \{{\rm per} (I-A)|A\in \tilde{\omega}_n^s\},
\end{equation*}
for every $0\leq s \leq n-1$. Moreover, the matrices in $\omega_n^s$ which makes the upper bound attained are also given. For $n-1<s\leq n$, the upper bound is known in the case when $n$ is even, but getting complicated and unclear when $n$ is odd. We will discuss the problem in the last section of this paper.

Another interesting characteristic on doubly substochastic matrices called {\it sub-defect} was first defined by Cao, Koyuncu and Parmer in~\cite{CK2016}. It is the smallest integer $k$ such that there exists an $(n + k) \times (n + k)$ doubly stochastic matrix containing $A$ as a submatrix. We often use $sd(A)$ to denote the sub-defect of a matrix $A\in\omega_n$. For more details about the definition, please see~\cite{CK2016}.
It has been shown by Theorem 2.1 in~\cite{CK2016} that the sub-defect can be calculated easily by taking the ceiling of the difference of its size and the sum of all entries. That is for $A\in\omega_n$, we have $$sd(A)=\lceil n-\sigma(A)\rceil.$$
Denote by $\omega_{n,k}$ the set of all $n\times n$ doubly substochastic matrices with sub-defect equal to $k$. Then we can partition $\omega_n$ into $n + 1$ convex subsets which are $\omega_{n,0} = \Omega_n, \omega_{n,1}, . . . , \omega_{n,n.}$ Namely,
\begin{enumerate}
\item $\omega_{n,k}$ is convex for all $k;$
\item $\omega_{n,i}\cap \omega_{n,j}=\emptyset$ for $i\neq j;$
\item $\bigcup_{i=0}^n\omega_{n,i}=\omega_n.$
\end{enumerate}
Also we see that
\begin{equation*}
\omega_{n,k}=\bigcup_{n-k\leq s<n-k+1} \{A\in\omega_n| \sigma(A)=s\}.
\end{equation*}
As a consequence, we obtain the upper bound of per$(I-A)$ with respect to the sub-defect of $A$ as well.

This paper is organized as follows. In section~\ref{sec:Pre}, we use Ryser's representation of permanent to show that if $A\in \omega_n$  which maximizes ${\rm per}(I-A)$, then all elements on the main diagonal of $A$ are zero. In section~\ref{sec:maxper}, we give the upper bound of ${\rm per}(I-A)$ for both $A\in\omega_n^s$ and $A\in\tilde{\omega}_n^s$ satisfying either $n$ even or $\sigma(A)\leq n-1$. For the case that $n$ is odd and $n-1<\sigma(A)$, the upper bound of ${\rm per}(I-A)$ for $A\in\tilde{\omega}_n^s$ is given in section~\ref{sec:ques}. For $A\in\omega_n^s$, the upper bound of ${\rm per}(I-A)$ still remains mystery. We discuss the case when $n$ is small and a few conjectures are also given in the last section.

\section{Preliminary}\label{sec:Pre}

In \cite{MM}, one of the main results is that if per$(I-A)$ is maximum for $A\in\tilde{\omega}_n$, then $A$ has a zero main diagonal. To show this, for a given $A\in\tilde{\omega}_n$ with non-zero diagonal elements, Malek in~\cite{MM} constructed another matrix $B\in \tilde{\omega}_n$ which has more zero elements on the main diagonal such that ${\rm per}(I-B)>{\rm per}(I-A).$ However, $\sigma(B) > \sigma(A)$ in Malek's construction, so we can not use the result. Here we give an alternative proof via the representation of the permanent given by Ryser (See Theorem 7.1.1 in \cite{CTM91}, also \cite{ryser_1963}). To state the representation, we need some notations.

Let $A$ be an $n\times n$ matrix. We denote by $r_i(A)$ the sum of all elements in the $i$th row of $A$. Let $A(j_1, j_2,\ldots, j_m)$ be the matrix obtained from $A$ by replacing the elements in columns $j_1, j_2, \ldots, j_m$ by zero's. Thus $r_i(A(j_1, j_2,\ldots, j_m))$ is the sum of all elements in the $i$th row of $A$ except for $a_{i j_1}, a_{i j_2}, \ldots, a_{i j_m}$. Let $S(A)=\prod_{i=1}^n r_i(A)$, the product of all row sums of $A$. Thus $S(A_m(j_1,j_2,\ldots, j_m))=\prod_{i=1}^n r_i(A(j_1,j_2,\ldots, j_m))$. Denote by $A_m$ the matrix obtained from $A$ by replacing some $m$ columns of $A$ by zero columns. We also use $\sum(-1)^mS(A_m)$ to denote the sum over all $n$ choose $m$ replacements of $m$ columns of $A$ by zero columns.

\begin{prop} \label{Pro1}
Let $A=[a_{ij}]\in \omega_n$ and $P=I-A$. Then
\begin{enumerate}
\item $r_i(P) \geq 0.$
\item $r_i(P(j_1,j_2,\ldots, j_m)) \leq 0 $ if $i=j_t$ for some $1\leq t\leq m$.
\item $S(P_m(j_1,j_2,\ldots, j_m)) \leq 0 $ if $m$ is odd, and $S(P_m(j_1,j_2,\ldots, j_m)) \geq 0 $ if $m$ is even.
\end{enumerate}
\end{prop}

\begin{proof}
\textit{(1)} It is because $r_i(P)=1-\sum_{j=1}^n a_{i j}\geq 0$.

\textit{(2)} Notice that in each row $i$ of $P$ for $1\leq i\leq n$, only the diagonal entry $1-a_{ii}$ is nonnegative. If $i=j_t$ for some $1\leq t\leq m$, that means the diagonal element $1-a_{i i}$ needed to be replaced by $0$ when we calculate $r_i(P(j_1, j_2, \ldots, j_m))$. Thus the $i$th row of $P(j_1,\ldots, j_m)$ only contains non-positive entries. Therefore $r_i(P(j_1,j_2,\ldots, j_m)) \leq 0$.

\textit{(3)} It follows from (2) because in $S(P_m(j_1,j_2,\ldots, j_m))$ there are $m$ non-positive factors, which are $r_{j_1}(P(j_1,j_2,\ldots, j_m)), \ldots, r_{j_m}(P(j_1,j_2,\ldots, j_m))$.
\end{proof}

\begin{thm}\label{Ryser}(Ryser's presentation of the permanent~\cite{CTM91, ryser_1963}) Let $A$ be a matrix, then
\begin{equation}\nonumber {\rm per}(A)=S(A_0)+\sum(-1)S(A_1)+\ldots+\sum (-1)^rS(A_r)+\ldots+\sum(-1)^{n-1}  S(A_{n-1}). \end{equation}
\end{thm}

\begin{lem} \label{leilm1}
Let $A=[a_{ij}]$ be an $n \times n$  matrix with all entries nonnegative and $a_{ii}\neq 0$ for some $i.$ Define a matrix $\tilde{A}=[\tilde{a}_{ij}]$ where all elements in $\tilde{A}$ are the same as elements in $A$ except
$$\tilde{a}_{ii}=a_{ii}-\epsilon,\  \tilde{a}_{ij}=a_{ij}+\epsilon$$
for some $j\neq i.$ Then
$${\rm per}(I-A)\leq per(I-\tilde{A}).$$
\end{lem}

\begin{proof}
Without loss of generality, we prove the case when $i=1$ and $j=2.$ Assume that $a_{11}\neq 0$ and let $$\tilde{a}_{11}=a_{11}-\epsilon,\  \tilde{a}_{12}=a_{12}+\epsilon,$$  then we need to show that
$${\rm per}(I-A)\leq {\rm per}(I-\tilde{A}).$$
Otherwise, one can always consider $PAQ$ instead of $A,$ where $P, Q$ are permutation matrices such that the element in the first row and first column of $PAQ$ is positive.

Denote
\begin{equation*}
M=I-A=\begin{pmatrix}
1-a_{1 1} & -a_{1 2} & \cdots & -a_{1 n}\\
 -a_{2 1} & 1-a_{2 2} & \cdots & -a_{2 n} \\
 \vdots &\vdots & \ddots & \vdots \\
-a_{n 1} & -a_{n 2} & \cdots & 1-a_{n n}
\end{pmatrix},
\end{equation*}
and
\begin{equation*}
N=I-\tilde{A}=\begin{pmatrix}
1-a_{1 1}+\epsilon & -a_{1 2}-\epsilon & \cdots & -a_{1 n}\\
 -a_{2 1} & 1-a_{2 2} & \cdots & -a_{2 n} \\
 \vdots &\vdots & \ddots & \vdots \\
-a_{n 1} & -a_{n 2} & \cdots & 1-a_{n n}
\end{pmatrix}.
\end{equation*}
In order to show that ${\rm per}(M)\leq {\rm per}(N),$ we apply Theorem \ref{Ryser} to both $M$ and $N$ and compare corresponding terms in \eqref{R1} and \eqref{R2}.
\begin{equation}\label{R1} {\rm per}(M)=S(M_0)+\sum(-1)S(M_1)+\ldots+\sum (-1)^rS(M_r)+\ldots+\sum(-1)^{n-1}  S(M_{n-1}) \end{equation}
\begin{equation}\label{R2} {\rm per}(N)=S(N_0)+\sum(-1)S(N_1)+\ldots+\sum (-1)^rS(N_r)+\ldots+\sum(-1)^{n-1}  S(N_{n-1}) \end{equation}
For all $1\leq j\leq n$, we have $r_j(M)=r_j(N)$, which implies $S(M_0)=S(N_0)$.
Notice that
\begin{equation*}
\sum S(M_1)=\sum_{i=1}^n \prod_{j=1}^n r_j(M(i)),
\end{equation*}
and
\begin{align*}
\sum S(N_1)=&\sum_{i=1}^n \prod_{j=1}^n r_j(N(i))\\
      =&(r_1(M(1))-\epsilon)r_2(M(1))\cdots r_n(M(1))\\
       &+(r_1(M(2))+\epsilon)r_2(M(2))\cdots r_n(M(2))+\sum_{i=3}^n \prod_{j=1}^n r_j(M(i))\\
      =&\sum S(M_1)+\epsilon[r_2(M(2))\cdots r_n(M(2))-r_2(M(1))\cdots r_n(M(1))],
\end{align*}
where $r_2(M(2))\leq 0$ and all other row sums are nonnegative due to Proposition \ref{Pro1} (ii). Thus we have
$$\sum[S(N_1)-S(M_1)]=\epsilon[r_2(M(2))\cdots r_n(M(2))-r_2(M(1))\cdots r_n(M(1))].$$
Therefore $\sum[S(N_1)-S(M_1)]\leq 0.$ We can write $\sum[S(N_1)-S(M_1)]=(-1)C_1$, where $C_1\geq 0$ a positive constant.

In general for $t\geq 2$, we have
\begin{equation*}
\sum S(M_t)=\sum_{1\leq i_1<i_2<\ldots < i_t\leq n} \prod_{j=1}^n r_j(M(i_1,i_2,\ldots, i_t)),
\end{equation*}
and
\begin{align*}
\sum S(N_t)=&\sum_{1\leq i_1<i_2<\ldots < i_t\leq n} \prod_{j=1}^n r_j(N(i_1,i_2,\ldots, i_t))\\
      =&\sum_{3\leq i_2<i_3<\ldots < i_t\leq n} \left[\prod_{j=1}^n r_j(N(1,i_2,\ldots, i_t))+\prod_{j=1}^n r_j(N(2,i_2,\ldots, i_{t}))\right]\\
       &+\sum_{3\leq i_3<i_4<\ldots < i_t\leq n} \prod_{j=1}^n r_j(N(1, 2, i_3,\ldots, i_t))+\sum_{3\leq i_1<i_2<\ldots < i_t\leq n} \prod_{j=1}^n r_j(N(i_1,i_2,\ldots, i_t))\\
      =&\sum_{3\leq i_2<i_3<\ldots < i_t\leq n} \left[\left(r_1(M(1, i_2, \ldots, i_t))-\epsilon\right)\prod_{j=2}^n r_j(M(1,i_2,\ldots, i_t))\right.\\
       &\left.+\left(r_1(M(2, i_2, \ldots, i_t))+\epsilon\right)\prod_{j=2}^n r_j(M(2,i_2,\ldots, i_{t}))\right] \\
       &+\sum_{3\leq i_3<i_4<\ldots < i_t\leq n} \prod_{j=1}^n r_j(M(1,2,i_3,\ldots, i_t))+\sum_{3\leq i_1<i_2<\ldots <i_t\leq n} \prod_{j=1}^n r_j(M(i_1,i_2,\ldots, i_t)).
\end{align*}
So
\begin{align*}
\sum[S(N_t)-S(M_t)]&=\sum_{3\leq i_2<i_3<\ldots < i_t\leq n}\epsilon\left[\prod_{j=2}^n r_j(M(2,i_2,\ldots, i_{t}))-\prod_{j=2}^n r_j(M(1,i_2,\ldots, i_{t}))\right].
\end{align*}
Notice that as $j$ runs from $2$ to $n$, $r_j(M(2,i_2,\ldots, i_{t})) \leq 0$ when $j=2,i_2,\ldots, i_{t}$. In the same while, $r_j(M(1,i_2,\ldots, i_{t})) \leq 0$ when $j=i_2,\ldots, i_{t}$. Thus we have
\begin{equation}\label{leieq1}
\sum[S(N_t)-S(M_t)]=(-1)^{t}C_t
\end{equation}
where $C_t\geq 0$ a constant. Now we can write
\begin{align*}
{\rm per}(N)-{\rm per}(M)&=S(N_0)-S(M_0)+\sum_{t=1}^{n-1}(-1)^t\left[S(N_t)-S(M_t)\right]\\
                         &=\sum_{t=1}^{n-1}(-1)^{2t}C_t =\sum_{t=1}^{n-1}C_t\geq 0.
\end{align*}
Therefore,
\begin{equation*}
{\rm per}(I-\tilde{A})\geq {\rm per}(I-A).
\end{equation*}
\end{proof}
\begin{cor} \label{Leico1}
Let $A=[a_{ij}]$ be an $n\times n$ row substochastic(stochastic) matrix satisfying $\sigma(A)=s$ and
\begin{align*}
{\rm per}(I-A)=\max\{{\rm per}(I-B)\ |\ &B\ {\rm is\ row\ substochastic(stochastic)}, \\
                                        &{\rm and}\ \sigma(B)=s \}.
\end{align*}
Then all the main diagonal entries of $A$ are zero.
\end{cor}
\begin{proof}
Notice that in the proof of Lemma~\ref{leilm1}, $\tilde{A}$ and $A$ have the same row sums and $\sigma(\tilde{A})=\sigma(A)$. Applying  Lemma~\ref{leilm1} as many times as possible we get the corollary.
\end{proof}
This corollary can be viewed as a refinement of Corollary~2 in~\cite{MM} since the matrices considered here have the same summation of all entries. Due to the different method in the proof of Corollary~\ref{Leico1}, we can also refine Proposition 1 in~\cite{MM} as the following lemma.
Denote by $\tilde{\omega}_n^s$ the set of all $n\times n$ row substochastic matrices such that the sum of all entries is equal to constant $s$.
\begin{lem} \label{Leilm2}
Let $A=[a_{ij}]$ be an $n\times n$ row substochastic matrix satisfying $\sigma(A)=s$ and
$${\rm per}(I-A)=\max \{{\rm per}(I-B)\ |\ B\in\tilde{\omega}_n^s \}.$$
We have
\begin{enumerate}
\item if there exist $a_{ki}\neq 0$ and $a_{kj}\neq 0$ for distinct indices $i,j$ and $k,$ then
$${\rm per}(A(k|i))={\rm per}(A(k|j)),$$
where $A(i|j)$ is the submatrix of $A$ obtained by removing the $i$th row and $j$th column from $A.$
\item there exists a row substochastic matrix $C$ such that
\begin{enumerate}
\item all diagonal elements are zero;
\item at most one positive entry contained in each row;
\item $\sigma(A)=\sigma(C)$ and
\item ${\rm per}(I-C)={\rm per}(I-A).$
\end{enumerate}
\end{enumerate}
\end{lem}

\begin{proof}
$(1)$ Assume that ${\rm per}(A(k|j)) \geq {\rm per}(A(k|i))$. Let $\epsilon=\min\{a_{k i}, a_{k j}\}$.
Define a matrix $B=[b_{ij}]$ to be the following row substochastic matrix:
$$b_{ki} = a_{ki}-\epsilon,\  b_{kj} = a_{kj} + \epsilon, $$
and $b_{ij} = a_{ij}$ otherwise. Then
$${\rm per}(I - B) = {\rm per}(I - A) + \epsilon({\rm per}(A(k|j)) - {\rm per}(A(k|i))).$$
Since ${\rm per}(I - A)$ is maximum and $ \epsilon> 0,$ we conclude that
$${\rm per}(A(k|i)) = {\rm per}(A(k|j)).$$

\noindent $(2)$\  It follows from Lemma~\ref{leilm1} and the result of part~$(1)$.
\end{proof}

\section{The maximum of the permanent of $I-A$}\label{sec:maxper}

In this section, we shall prove the following theorem.
\begin{thm}\label{main1}
Let $A\in \omega_n$, satisfying either
\begin{enumerate}
\item $n$ is even, or
\item $n$ is odd and $\sigma(A)\leq n-1$.
\end{enumerate}
Let $\sigma(A)=s$ and denote by $e$ the greatest even integer less than or equal to $s$. Then
\begin{equation*}
\max \{{\rm per}(I-A)| A\in\omega_n^s \}= 2^{e/2}\left[1+\left(\frac{s-e}{2}\right)^2\right].
\end{equation*}
\end{thm}
Theorem~\ref{main1} can also be rephrased with respect to the sub-defect $k$ as the following corollary.
\begin{cor}
Let $A\in \omega_{n,k}$, where either
\begin{enumerate}
\item $n$ is even, or
\item $n$ is odd and $k>1$.
\end{enumerate}
Denote by $e$ the greatest even integer less than or equal to $n-k+1$. Then
\begin{equation*}
\sup \{{\rm per}(I-A)|A\in\omega_{n,k}\}= 2^{e/2}\left[1+\left(\frac{n-k+1-e}{2}\right)^2\right].
\end{equation*}
\end{cor}


For any $n$-square matrix $A$, we can always associate with $A$ a directed weighted graph $G_A$ with $n$ vertices $\{v_1, v_2, \ldots, v_n\}$, such that $a_{i j}\neq 0$ if and only if there is a directed edge (vector) $\overrightarrow{V_i V_j}$ with weight $a_{i j}$. Let $\tilde{\omega}_n^{0,1}$ be the set of all $n\times n$ row substochastic matrices with zero diagonal and at most one positive element in each row. For $A\in\tilde{\omega}_n^{0,1}$, since each row of $A$ contains at most one positive element, there is at most one directed edge starting from each vertex in $G_A$. Consider the matrix $P=I-A$, the corresponding directed graph $G_P$ can be obtained from $G_A$ by adding a negative sign before each weight and also adding a ``loop" $\overrightarrow{V_i V_i}$ with weight $1$ at each vertex $V_i$. From definition, we know that
\begin{equation*}
{\rm per} P = {\rm per} (I-A)=\sum_{\pi\in S_n} p_{1\pi(1)}p_{2\pi(2)}\cdots p_{n\pi(n)},
\end{equation*}
where the sum is over all permutations $\pi\in S_n$. Since each permutation can also be written in cycle notation, we can write
\begin{equation}
\pi=(i_1 i_2\cdots i_{r_1})(i_{r_1+1} i_{r_1+2}\cdots i_{r_2})\cdots (i_{r_{l-1}+1} i_{r_{l-1}+2}\cdots i_{r_{l(\pi)}}),
\end{equation}
where $l=l(\pi)$ for convenience, and thus
\begin{equation*}
{\rm per} (I-A)=\sum_{\pi\in S_n} (p_{i_1 i_2}p_{i_2 i_3}\cdots p_{i_{r_1} i_1})
                                  (p_{i_{r_1+1} i_{r_1+2}}\cdots p_{i_{r_2} i_{r_1+1}})
                                  \cdots (p_{i_{r_{l-1}+1} i_{r_{l-1}+2}}\cdots p_{i_{r_{l(\pi)}} i_{r_{l-1}+1}}).
\end{equation*}
Therefore for $0\leq t\leq l(\pi)-1$ and in case $t=0$ setting $r_0=0$, each cycle $(i_{r_t+1} i_{r_t+2} \cdots i_{r_{t+1}})$ in permutation $\pi$ corresponds to a directed cycle in $G_P$:
\begin{equation*}
\xymatrix{
V_{i_{r_t+1}} \ar[rr]^{p_{i_{r_t+1} i_{r_t+2}}} && V_{i_{r_t+2}} \ar[r] & \cdots \ar[r] & V_{i_{r_{t+1}}} \ar[rr]^{p_{i_{r_{t+1}} i_{r_{t}+1}}} && V_{i_{r_{t}+1}} }.
\end{equation*}
Clearly, the term $(p_{i_1 i_2}p_{i_2 i_3}\cdots p_{i_{r_1} i_1})(p_{i_{r_1+1} i_{r_1+2}}\cdots p_{i_{r_2} i_{r_1+1}})
\cdots (p_{i_{r_{l-1}+1} i_{r_{l-1}+2}}\cdots p_{i_{r_{l(\pi)}} i_{r_{l-1}+1}})$
does not vanish if and only if the weights in each directed circle are all positive. Now we can prove the following theorem.
\begin{lem}
For $A\in\tilde{\omega}_n^{0,1}$ such that the corresponding directed graph $G_A$ has only one connected component, then
\begin{enumerate}
\item ${\rm per}(I-A)=1$ if there is no directed cycle in $G_A$;

\item ${\rm per}(I-A)=1+(-1)^{r}a_{i_1 i_2}a_{i_2 i_3}\cdots a_{i_r i_1}$ if there is a cycle with length $r$ and $a_{i_1 i_2}, a_{i_2 i_3}, \ldots, a_{i_r i_1}$ are weights of edges in the cycle
\begin{equation*}
\xymatrix{V_{i_{1}} \ar[rr]^{a_{i_{1} i_{2}}} && V_{i_{2}} \ar[r] & \cdots \ar[r] & V_{i_{r}} \ar[rr]^{a_{i_{r} i_{1}}} && V_{i_{1}} }.
\end{equation*}
\end{enumerate}
\end{lem}
\begin{proof}
Since $A\in\tilde{\omega}^{0,1}$ has at most one positive element in each row, there is at most one vector pointing out from each vertex in $G_A$. Therefore every connected component contains either no cycles or at most one directed cycle. If there is no cycles in $G_A$, then $G_{I-A}$ contains no nontrivial directed cycles. Thus there is only one nonzero term in the expansion of ${\rm per}(I-A)$, which corresponds to $id\in S_n$, i.e., all loops in $G_{I-A}$. So in this case ${\rm per}(I-A)=1$. Suppose $G_A$ contains one directed cycle $C$ with $r$ edges. The weights in the cycle are $a_{i_1 i_2}, a_{i_2 i_3}, \ldots, a_{i_r i_1}$. Then there are two nonzero terms in the expansion of ${\rm per}(I-A)$ which correspond to $id\in S_n$ and $(i_1 i_2 \cdots i_r)$ respectively. Therefore in this case
\begin{equation*}
{\rm per}(I-A)=1+(-1)^r a_{i_1 i_2}a_{i_2 i_3}\cdots a_{i_r i_1}.
\end{equation*}
\end{proof}
For $C$ a cycle in $G_A$ consisting of $r$ edges with weights $a_{i_1 i_2}, a_{i_2 i_3},\ldots, a_{i_r i_1}$, denote the length of $C$ by $l(C)$ and then $l(C)=r$. Let $1+(-1)^{l(C)}C=1+(-1)^r a_{i_1 i_2}a_{i_2 i_3}\ldots a_{i_r i_1}$. We give some results of ${\rm per}(I-A)$ for general $A\in\tilde{\omega}_n^{0,1}$.
\begin{lem}\label{lm:perf}
Let $A\in\tilde{\omega}_n^{0,1}$. If $G_A$ contains a cycle $C$ which consists of edges with weights $a_{i_1 i_2}, a_{i_2 i_3},\ldots, a_{i_r i_1}$, then
\begin{align*}
{\rm per}(I-A)&={\rm per}(I-\tilde{A})\left(1+(-1)^r a_{i_1 i_2}a_{i_2 i_3}\cdots a_{i_r i_1}\right)\\
              &={\rm per}(I-\tilde{A})(1+(-1)^{l(C)}C),
\end{align*}
where $\tilde{A}$ is the submatrix obtained by removing the $r$ rows and columns containing $a_{i_1 i_2}, a_{i_2 i_3},\ldots, a_{i_r i_1}$.
\end{lem}
\begin{proof}
From the definition of ${\rm per}(I-A)$ we know that the nonzero terms in the expansion correspond to permutations haing cycles either $(i_1 i_2 \cdots i_r)$ or $(i_1)(i_2)\cdots (i_r)$. Therefore the lemma holds.
\end{proof}
For $A\in\tilde{\omega}_n^{0,1}$, by removing the rows and columns containing $a_{i_1 i_2}, a_{i_2 i_3},\ldots, a_{i_r i_1}$ we get the submatrix $\tilde{A}$, which is in $\tilde{\omega}_{n-r}^{0,1}$. Thus keep applying Lemma~\ref{lm:perf} and we get the following lemma which expresses ${\rm per} (I-A)$ into the product of factors $(1+(-1)^{l(C)}C)$.
\begin{lem}\label{lm:perprof}
Let $A\in\tilde{\omega}_n^{0,1}$ whose corresponding graph contains $k$ cycles $C_1, C_2, \ldots, C_k$ in total, then
\begin{equation}\label{eq:perprod}
{\rm per}(I-A)=\prod_{i=1}^{k}\left(1+(-1)^{l(C_i)}C_i\right).
\end{equation}
\end{lem}
\begin{lem}\label{lm:neqper}
Let $A$ be an $n$-square row substochastic matrix with at most one positive entry contained in each row.
\begin{enumerate}
\item If $A$ contains even number of positive elements, then
\begin{equation*}
{\rm per}(I-A)\leq (1+x_1x_2)(1+x_3x_4)\cdots(1+x_{2t-1}x_{2t})
\end{equation*}
where $x_1, x_2, \ldots, x_{2t}$ is a labeling of the even positive elements in $A$.
\item If $A$ contains odd number of positive elements, then
\begin{equation*}
{\rm per}(I-A)< (1+x_1x_2)(1+x_3x_4)\cdots(1+x_{2t-1}x_{2t})(1+\frac{x_{2t+1}^2}{4})
\end{equation*}
where $x_1, x_2, \ldots, x_{2t} , x_{2t+1}$ is a labeling of the odd positive elements in $A$.
\end{enumerate}
\end{lem}
\begin{proof}
Since by Lemma~\ref{lm:perprof}, we can write ${\rm per}(I-A)$ into products as in equation~\eqref{eq:perprod}.
We first label the elements appearing in~\eqref{eq:perprod} by $x_1, x_2,\ldots$, and then we label the remaining positive elements left in $A$.
For example we can relabel the sequence $a_{i_1 i_2}, a_{i_2 i_3}, a_{i_3 i_4}, \ldots, a_{i_r i_1}$ appearing in one factor $(1+(-1)^r a_{i_1 i_2}a_{i_2 i_3}\cdots a_{i_r i_1})$ by $x_1, x_2, \cdots, x_r$, respectively. Such a sequence forms a factor $(1+(-1)^rx_1x_2 \cdots x_r)$ in ${\rm per}(I-A)$. If $r$ is even, then
\begin{align*}
1+(-1)^rx_1x_2 \cdots x_r\leq (1+x_1x_2)(1+x_3x_4)\cdots(1+x_{r-1}x_r).
\end{align*}
If $r$ is odd, then
\begin{align*}
1+(-1)^rx_1x_2 \cdots x_r &=1-x_1x_2 \cdots x_r \\
                          &<(1+x_1x_2)(1+x_3x_4)\cdots(1+x_{r-2}x_{r-1})\\
                          &<(1+x_1x_2)(1+x_3x_4)\cdots(1+x_{r-2}x_{r-1})(1+\frac{x_r^2}{4}).
\end{align*}
Also notice that for $r$ and $r'$ odd, we have
\begin{align*}
&(1+(-1)^rx_1x_2 \cdots x_r)(1+(-1)^rx_{r+1}x_{r+2} \cdots x_{r+r'})\\
                           &<(1+x_1x_2)\cdots(1+x_{r-2}x_{r-1})(1+x_{r}x_{r+1})\cdots(1+x_{r+r'-1}x_{r+r'}).
\end{align*}
Therefore the lemma holds.
\end{proof}

\begin{cor}\label{cor:maxper}
Let $A$ be an $n$-square row substochastic matrix with at most one positive entry contained in each row.
\begin{enumerate}
\item If $A$ contains even number of positive elements, and $x_1, x_2, \ldots, x_{2t}$ is a labeling of the $2t$ positive elements in $A$ for $0\leq 2t\leq n$, then $A$ can be permutated by some permutation matrices to the following form
\begin{equation}\label{eq2}
\begin{pmatrix}0 & x_1 \\ x_2 & 0
\end{pmatrix}\oplus
\begin{pmatrix}0 & x_3 \\ x_4 & 0
\end{pmatrix}\oplus\ldots\oplus\begin{pmatrix} 0 & x_{2t-1} \\ x_{2t} & 0 \end{pmatrix}\oplus \mathbf{0}_{n-2t}.
\end{equation}

\item If $A$ contains odd number of positive elements, and $x_1, x_2, \ldots, x_{2t} , x_{2t+1}$ is a labeling of the $2t+1$ positive elements in $A$ for $1\leq 2t+1 \leq n-1$, then we can construct another $n$-square row substochastic matrix $\tilde{A}$ satisfying $\sigma(\tilde{A})=\sigma(A)$. Also $\tilde{A}$ can be permutated by some permutation matrices to the following form
\begin{equation}\label{eq3}
\begin{pmatrix}0 & x_1 \\ x_2 & 0  \end{pmatrix}\oplus\begin{pmatrix}0 & x_3 \\ x_4 & 0  \end{pmatrix}\oplus\ldots\oplus\begin{pmatrix}0 & x_{2t-1} \\ x_{2t} & 0  \end{pmatrix}\oplus \begin{pmatrix}0 & x_{2t+1}/2 \\ x_{2t+1}/2 & 0  \end{pmatrix}\oplus \mathbf{0}_{n-2t-2}
\end{equation}
such that
\begin{equation*}
{\rm per}(I-A)< {\rm per}(I-\tilde{A}).
\end{equation*}
\end{enumerate}
\end{cor}
\begin{proof}
(1) It is not difficult to see that the matrix in~\eqref{eq2} maximizes the value of ${\rm per}(I-A)$, which is equal to $(1+x_1x_2)(1+x_3x_4)\cdots(1+x_{2t-1}x_{2t})$.

(2) Substituting~\eqref{eq3} to ${\rm per}(I-\tilde{A})$ we get
\begin{equation*}
{\rm per}(I-\tilde{A})=(1+x_1x_2)(1+x_3x_4)\cdots(1+x_{2t-1}x_{2t})(1+\frac{x_{2t+1}^2}{4}).
\end{equation*}
By Lemma~\ref{lm:neqper} (2) we know that
${\rm per}(I-A)< {\rm per}(I-\tilde{A})$.
\end{proof}
We can then get a property of row substochastic matrices.
\begin{cor}
For $n$ even and $A$ an $n$-square row substochastic matrix, we have
\begin{equation}\label{eq1}{\rm per}(I-A)\leq (1+x_1x_2)(1+x_3x_4)\cdots(1+x_{n-1}x_n)\end{equation} where $x_1, x_2, \ldots, x_n$ is a labeling of the row sums of $A$.
\end{cor}
\begin{proof}
It follows from Lemma~\ref{Leilm2} and Lemma~\ref{lm:neqper}.
\end{proof}
\begin{prop}\label{prop2}
Let $A,B$ be square matrices, then
$${\rm per}\begin{pmatrix}A & \\ & B\end{pmatrix}={\rm per}(A){\rm per}(B)$$
\end{prop}
\begin{proof}
It follows from the definition of the permanent.
\end{proof}

\begin{prop}\label{prop3}
Let $x, y$ be non-negative numbers and the sum of $x$ and $y$ is fixed.
Then
\begin{equation*}
\max {\rm per}\begin{pmatrix} 1 & x \\ y & 1\end{pmatrix}=\max(1+xy)=1+\left(\frac{x+y}{2}\right)^2.
\end{equation*}
\end{prop}
\begin{proof}
It follows from the arithmetic-geometric  inequality that $xy\leq (\frac{x+y}{2})^2.$
\end{proof}

\begin{lem}\label{lmzhi}
For a sequence satisfying $0< z_n\leq z_{n-1}\leq \ldots \leq z_2\leq z_1 <1$ with $\sum_{i=1}^nz_i=\bar{s}$ fixed,  let $\epsilon=\min\{1-z_1,z_n\}$. Define  $$y_1=z_1+\epsilon, y_n=z_n-\epsilon$$ and $y_i=z_i$ for $i=1,2,\ldots,n.$ Then
\begin{equation}\label{mainr}
\sum_{k=1}^n\sum_{1\leq i_1<i_2<\ldots<i_k\leq n}z_{i_1}^2 z_{i_2}^2\cdots z_{i_k}^2 < \sum_{k=1}^{n}\sum_{1\leq i_1<i_2<\ldots<i_k\leq n}y_{i_1}^2 y_{i_2}^2\cdots y_{i_k}^2.
\end{equation}
\end{lem}
\begin{proof} First we show \eqref{mainr} holds when $n=2.$
In this case $y_1=z_1+\epsilon\leq 1$ and $y_2=z_2-\epsilon \geq 0.$ We need to prove that
$$y_1^2+y_2^2+y_1^2y_2^2 -(z_1^2+z_2^2+z_1^2z_2^2)>0,$$
which is equivalent to
\begin{equation}\label{main}
(z_1+\epsilon)^2+(z_2-\epsilon)^2+(z_1+\epsilon)^2(z_2-\epsilon)^2-(z_1^2+z_2^2+z_1^2z_2^2)>0.
\end{equation}
Since $\epsilon = \min \{ 1-z_1, z_2\},$ there are two possibilities: either $\epsilon=1-z_1$ or $\epsilon=z_2$. We discuss the two cases separately as follows.
\begin{enumerate}
\item If $\epsilon=1-z_1,$ then $y_1=1$ and $y_2=z_1+z_2-1 \geq 0.$ Thus we have
\begin{eqnarray}\nonumber
&& y_1^2+y_2^2+y_1^2y_2^2-(z_1^2+z_2^2+z_1^2 z_2^2) \\
\nonumber &=& 1+(z_1+z_2-1)^2+(z_1+z_2-1)^2-z_1^2-z_2^2-z_1^2z_2^2\\
\nonumber &=& 3+z_1^2+z_2^2+4z_1z_2-4z_1-4z_2-z_1^2z_2^2\\
\nonumber &=& (z_1^2+z_2^2-z_1^2z_2^2-1)+4(z_1z_2-z_1-z_2+1)\\
\nonumber &=&-(1-z_1^2)(1-z_2^2)+4(1-z_1)(1-z_2)\\
 \label{eqc1}&=&(1-z_1)(1-z_2)[4-(1+z_1)(1+z_2)]
\end{eqnarray}
By the assumption that $0<z_2\leq z_1 <1,$ all factors in \eqref{eqc1} are positive, so \eqref{main} holds.

\item  If $\epsilon=z_2,$ then $y_2=0$ and $y_1=z_1+z_2\leq 1.$ Thus we have
\begin{eqnarray}\nonumber
&& y_1^2+y_2^2+y_1^2\cdot y_2^2-(z_1^2+z_2^2+z_1^2 z_2^2) \\
\nonumber &=& (z_1+z_2)^2-z_1^2-z_2^2-z_1^2z_2^2\\
\nonumber &=& z_1^2z_2^2> 0
\end{eqnarray}
 So \eqref{main} holds.
\end{enumerate}

To show \eqref{mainr} in general cases, notice that
\begin{align}
&\sum_{k=1}^{n}\sum_{1\leq i_1<i_2<\ldots<i_k\leq n}y_{i_1}^2 y_{i_2}^2\cdots y_{i_k}^2 -\sum_{k=1}^n\sum_{1\leq i_1<i_2<\ldots<i_k\leq n}z_{i_1}^2 z_{i_2}^2\cdots z_{i_k}^2 \nonumber \\
=&\left(y_1^2+y_n^2+y_1^2y_n^2-z_1^2-z_n^2-z_1^2z_n^2\right)\left(1+\sum_{k=1}^{n-2}\sum_{2\leq i_1<i_2<\ldots<i_k\leq n-1}\prod_{m=1}^k z_{i_m}^2\right).\label{eqc2}
\end{align}
The second factor on the righthand side of equation~\eqref{eqc2} is obviously positive.
Due to \eqref{main}, we know that
\begin{align*}
& y_1^2+y_n^2+y_1^2y_n^2-z_1^2-z_n^2-z_1^2z_n^2 \\
=&(z_1+\epsilon)^2+(z_n-\epsilon)^2+(z_1+\epsilon)^2(z_n-\epsilon)^2-\left(z_1^2+z_n^2+z_1^2z_n^2\right)>0.
\end{align*}
Therefore \eqref{eqc2} is strictly greater than zero and \eqref{mainr} is proved.
\end{proof}

\begin{lem}\label{Leilm5}
Let $0\leq z_n\leq z_{n-1}\leq \ldots \leq z_2\leq z_1 \leq 1$ satisfying $\sum_{i=1}^nz_i=\bar{s}.$ Then
\begin{equation} \label{Leieq3}
\max \left(1+\sum_{k=1}^n\sum_{1\leq i_1<i_2<\ldots<i_k\leq n}z_{i_1}^2 z_{i_2}^2\cdots z_{i_k}^2\right)=2^{\lfloor \bar{s}\rfloor}\left[1+(\bar{s}-\lfloor \bar{s}\rfloor)^2\right],
\end{equation}
where $\lfloor \bar{s}\rfloor$ denotes the greatest integer less than or equal to $\bar{s}$.
\end{lem}
\begin{proof}
First if we let $y_1=y_2=\ldots=y_{\lfloor \bar{s}\rfloor}=1, y_{\lfloor \bar{s}\rfloor+1}=\bar{s}-\lfloor \bar{s}\rfloor$, and $y_i=0$ for $\lfloor \bar{s}\rfloor+1<i\leq n$, then
$$1+\sum_{k=1}^n\sum_{1\leq i_1<i_2<\ldots<i_k\leq n}y_{i_1}^2 y_{i_2}^2\cdots y_{i_k}^2 =2^{\lfloor \bar{s}\rfloor}\left[1+(\bar{s}-\lfloor \bar{s}\rfloor)^2\right].$$
Next we need to show it is the maximum. Indeed, suppose there exists $0< z_n\leq z_{n-1}\leq \ldots \leq z_2\leq z_1 \leq 1$ with $\sum_{i=1}^mz_i=\bar{s}$.  According to Lemma~\ref{lmzhi}, let $\epsilon=\min\{1-z_r, z_m\}$ and we can add $\epsilon$ to some $z_r$ where $r$ is the smallest index such that $z_r<1$. We then subtract $\epsilon$ from $z_n$. This makes either $z_r$ be $1$ or $z_n$ be $0$, and the sum in~\eqref{Leieq3} greater than before without changing the sum of $z_i$'s. Repeat this process as many times as possible until we cannot do it any more. Eventually the sequence $\{z_i\}$ will be changed into $\{y_i\}$ and then the corresponding value attained is maximum.
\end{proof}
Now, we are ready to prove Theorem~\ref{main1} which is a direct consequence of the following theorem.
\begin{thm}\label{thm:premain}
For $A$ an $n$-square row substochastic matrix with $\sigma(A)=s$. Denote by $e$ the greatest even integer less than or equal to $s$. If $n$ and $s$ satisfying either
\begin{enumerate}
\item $n$ is even, or
\item $n$ is odd and $s\leq n-1,$
\end{enumerate}
then
\begin{equation*}
\max \{{\rm per}(I-A)| A\in\tilde{\omega}_n^s\}= 2^{e/2}\left[1+\left(\frac{s-e}{2}\right)^2\right].
\end{equation*}
\end{thm}
\begin{proof}
First we consider the case when $n$ is even.  According to Lemma~\ref{Leilm2} and~\ref{cor:maxper}, to maximize the value of ${\rm per}(I-A)$, $A$ must be in the form~\eqref{eq2}. Due to Proposition~\ref{prop2},
$${\rm per}(I-A)=\prod_{i=1}^{n/2}{\rm per}(I_2-X_i)$$
where $X_i=\begin{pmatrix}0& x_{2i-1}\\ x_{2i} &0 \end{pmatrix}$ for $i=1,2,\ldots, n/2.$ By direct calculation and the arithmetic-geometric inequality,
$${\rm per}(I_2-X_i)=1+x_{2i-1}x_{2_i}\leq 1+\left(\frac{x_{2i-1}+x_{2i}}{2}\right)^2.$$
In order to maximize per$(I_2-X_i),$ we should let $x_{2i-1}=x_{2i}$ due to Proposition~\ref{prop3}. Hence
\begin{eqnarray} \nonumber
{\rm per}(I-A)&=&\prod_{i=1}^{n/2}(1+x_{2i}^2)=\prod_{i=1}^{n/2}(1+y_{i}^2)\\
\label{eq4} &=&1+\sum_{k=1}^{n/2}\sum_{1\leq i_1<i_2<\ldots<i_k\leq n/2}y_{i_1}^2 y_{i_2}^2\cdots y_{i_k}^2,
\end{eqnarray}
where $y_{i}=x_{2i}$ for $i=1,2,\ldots, n/2$.
Apply Lemma~\ref{Leilm5} to~\eqref{eq4}, we get
$$\max \{{\rm per}(I-A)| A\in\tilde{\omega}_n^s\}=2^{e/2}\left[1+\left(\frac{s-e}{2}\right)^2\right].$$
In this case,
$$y_{1}=y_{2}=\ldots=y_{e/2}=1, y_{\frac{e}{2}+1}=\frac{s-e}{2},$$
and $$y_{\frac{e}{2}+2}=\cdots=y_{n/2}=0.$$
That means we can actually choose a doubly substochastic matrix $\tilde{A}$ with $\sigma(\tilde{A})=s$ as follows
$$\tilde{A}=M_2\oplus M_2 \oplus \ldots \oplus M_2 \oplus S_2 \oplus \mathbf{0}_{n-e-2},$$
where $M_2=\begin{pmatrix}0 & 1 \\ 1 & 0\end{pmatrix}$ with $e/2$ copies in total, $S_2=\begin{pmatrix}0 & \frac{s-e}{2} \\  \frac{s-e}{2} & 0 \end{pmatrix},$ and $\mathbf{0}_{n-e-2}$ is the zero matrix with order $n-e-2$. If $e=n$ then $\mathbf{0}_{n-e-2}$ won't show up.
It is easy to see that such an $\tilde{A}$ maximize the value ${\rm per}(I-A)$, where $A$ can be any $n$-square row substochastic matrix satisfying $\sigma(A)=s$.

In the case when $n$ is odd and $s\leq n-1$, by Corollary~\ref{cor:maxper} we can always construct a row substochastic matrix $B$ with row sum equals to $s$, such that $B$ contains even number of positive elements with at most one positive element in each row. Actually from  Corollary~\ref{cor:maxper} we can see that $B$ is also doubly substochastic. Using the similar method in proving the above case when $n$ is even, we get the result of the theorem. Notice that here $B$ takes the same form as $\tilde{A}$ except that $n$ is odd.
\end{proof}

\begin{remark}
According to the proof of Theorem~\ref{thm:premain}, both $\tilde{A}$ and $B$ are doubly substochastic, which maximize the value of ${\rm per}(I-A)$. Therefore Theorem~\ref{main1} follows from Theorem~\ref{thm:premain} naturally.
\end{remark}
\noindent {\bf Example}
Denote by $\mathbf{0}_n$ the $n\times n$ zero matrix and $M_2=\begin{pmatrix} 0 & 1 \\ 1 & 0\end{pmatrix}.$ If $A\in \omega_{9}$ with $\sigma(A)=5$,
then one can construct $$\tilde{A}=M_2\oplus M_2\oplus\frac{1}{2}M_2 \oplus \mathbf{0}_3,$$ which maximizes per$(I-A)$ for all $A$ doubly substochastic with fixed total sum $5$. It is easy to calculate that
$${\rm per}(I-\tilde{A})=2\cdot 2 \cdot (1+\frac{1}{4})=5.$$

\section{Further questions}\label{sec:ques}
The conditions in Theorem~\ref{main1} require that either $n$ is even or $\sigma(A)\leq n-1$, which leaves the case that $n$ is odd and $\sigma(A)>n-1$ uncovered. The requirement is due to the way provided in Corollary~\ref{main} to construct the doubly substochastic matrix which maximizes ${\rm per}(I-A).$
It is worth to point out that for $A\in\tilde{\omega}_n^s$, the maximum of ${\rm per}(I-A)$ can be easily obtained from Lemma~\ref{Leilm2}, Lemma~\ref{lm:perprof} and Lemma~\ref{lm:neqper}. We state the result as the following theorem.
\begin{thm}
For $n$ odd and $n-1<\sigma(A)\leq n$, let $\sigma(A)=s$ and then we have
\begin{equation*}
\max\{{\rm per}(I-A)|A\in\tilde{\omega}_n^s, n-1<s\leq n\}=2^{\frac{n-1}{2}}.
\end{equation*}
To get the maximum value, we can simply take
$$\tilde{A}=M_2\oplus M_2 \oplus \ldots \oplus M_2 \oplus M_3,$$
where $M_2=\begin{pmatrix}0 & 1 \\ 1 & 0\end{pmatrix}$ with $(n-3)/2$ copies in total, and
$M_3=\begin{pmatrix}0 & 1 & 0 \\ 1& 0 & 0 \\ 0 & s-(n-1) & 0  \end{pmatrix}$.
\end{thm}

Notice that the above $\tilde{A}$ is a row substochastic matrix but not a doubly substochastic matrix since the second column sum of $M_3$ is strictly greater than one. Thus the question that finding the maximum value of ${\rm per}(I-A)$ for $A\in\omega_n^s$ where $n$ is odd and $n-1<s\leq n$ becomes particularly difficult. In this section, we explore the special case when $n=s=3$, which is for all $A\in\Omega_3$. Then we give some conjectures based on this result.
\begin{lem}
$$\max_{A\in\Omega_3}\{{\rm per}(I-A)\}=\frac{3}{2}.$$
\end{lem}

\begin{proof} Suppose $A_0\in \Omega_3$ such that per$(I-A_0)$ is the maximum. Due to Corollary~\ref{Leico1}, $A_0$ must have zero diagonal. Since $A_0$ is also doubly stochastic, we can assume that $A_0$ has the following form
$$A_0=\begin{pmatrix}0 & x & 1-x \\1-x & 0 & x \\  x & 1-x & 0 \end{pmatrix}.$$
By direct computation we have
\begin{align*}
{\rm per}(I-A_0)&={\rm per}\begin{pmatrix}1 & -x & -1+x \\-1+x & 1 & -x \\  -x & -1+x & 1 \end{pmatrix}\\
              &=6x(1-x),
\end{align*}
which takes the maximum value $\frac{3}{2}$ when $x=\frac{1}{2}$.
That is equivalent to say
$$A_0=\begin{pmatrix}0& \frac{1}{2} & \frac{1}{2} \\ \frac{1}{2} & 0 & \frac{1}{2} \\ \frac{1}{2} & \frac{1}{2} & 0 \end{pmatrix},$$
and
$$\max_{A\in\Omega_3}\{{\rm per}(I-A)\}={\rm per}(I-A_0)=\frac{3}{2}.$$
\end{proof}
We then make the following conjecture.
\begin{conj}
Let $n$ be a positive odd integer. Then
$$\max\{{\rm per}(I-A)|A\in \Omega_n\}=2^{\frac{n-1}{2}}\cdot 3.$$
The maximum can be obtained by letting
$$A=M_2\oplus M_2 \oplus \ldots \oplus M_2 \oplus M_3,$$
where $M_2=\begin{pmatrix}0 & 1\\ 1& 0\end{pmatrix}$ with $\frac{n-1}{2}$ copies, and $M_3=\begin{pmatrix}0 & \frac{1}{2} & \frac{1}{2} \\ \frac{1}{2} & 0 & \frac{1}{2} \\ \frac{1}{2} & \frac{1}{2} & 0 \end{pmatrix}.$
\end{conj}

\begin{lem} Let $A\in \omega_3$ and $ 2<\sigma(A)\leq 3,$ then
$$\max_{A\in \omega_3}\{{\rm per}(I-A)\}\geq \max \{\frac{\sigma^2(A)-5\sigma(A)+12}{4}, 6-2\sigma(A)\}.$$
\end{lem}
\begin{proof}
Let $$A_0=\begin{pmatrix}0 & \frac{1}{2} & \frac{\sigma(A)}{2}-1 \\ \frac{1}{2} & 0 & \frac{1}{2} \\ \frac{\sigma(A)}{2}-1 & \frac{1}{2} & 0 \end{pmatrix}$$
and
$$A_1=\begin{pmatrix} 0 & 1 & 0 \\ 1 & 0 & 0 \\ 0 & 0 & \sigma(A)-2 \end{pmatrix}.$$
Thus we have
$${\rm per}(I-A_0)=\frac{\sigma^2(A)-5\sigma(A)+12}{4},$$
and
$${\rm per}(I-A_1)=6-2\sigma(A).$$
\end{proof}


\begin{conj}
 $$\max_{A\in \omega_3}\{{\rm per}(I-A)\}=\left\{
\begin{aligned}
&\frac{\sigma^2(A)-5\sigma(A)+12}{4} & \ {\rm if}\ \frac{-3+\sqrt{57}}{2}<\sigma(A)\leq 3\\
&6-4\sigma(A)                        & \ {\rm if}\  2 < \sigma(A)\leq \frac{-3+\sqrt{57}}{2}.
\end{aligned}
\right.$$
\end{conj}

\begin{conj}
Let $n$ be odd and $A\in \omega_n$ with $n-1<\sigma(A)\leq n.$ Denote $\sigma(A)$ by $s$. Then
$$\max\{{\rm per}(I-A)|A\in \omega_n^s\}=2^{\frac{n-1}{2}}\cdot c$$
where $\displaystyle c=\max\{{\rm per}(I-B)|B\in \omega_3^{s'}\}$ and $s'=s-n+3.$
The maximum can be obtained by letting $$A=M_2\oplus M_2\oplus \ldots \oplus M_2 \oplus B_3$$ where $M_2=\begin{pmatrix}0& 1\\ 1& 0\end{pmatrix}$ with $\frac{n-1}{2}$ copies, and $B_3$ is the $3\times 3$ matrix maximizing {\rm per}$(I-B)$ for all $B\in\omega_{3}^{s'}.$
\end{conj}

\end{document}